\documentclass[12pt,a4paper]{article}
\pdfoutput=1
\usepackage{amsbsy,amsfonts,amssymb,amsmath,amsthm}
\usepackage{graphicx}
\usepackage[left=2cm,right=2cm,top=2cm,bottom=2cm]{geometry}
\theoremstyle{plain}
\newtheorem{theorem}{Theorem}
\newtheorem{lemma}{Lemma}
\newtheorem{proposition}{Proposition}
\newtheorem{corollary}{Corollary}
\theoremstyle{definition}

\newtheorem{remark}{Remark}

\newcommand{\AAP}{\emph{Adv. Appl. Prob.}}
\newcommand{\BAMS}{\emph{Bull. Amer. Math. Soc.}}
\newcommand{\AP}{\emph{Ann. Prob.}}

\newcommand{\R}{\mathbb{R}}

\DeclareMathOperator{\card}{card}
\DeclareMathOperator{\spa}{span}

\DeclareMathOperator{\Var}{Var}
\allowdisplaybreaks 

\begin{document}
\author{Julia H\"orrmann\footnotemark[1]\,, Daniel Hug\footnotemark[2]}
\title{On the volume of the zero cell of a class of \\ isotropic Poisson hyperplane tessellations}
\date{}
\renewcommand{\thefootnote}{\fnsymbol{footnote}}
\footnotetext[1]{Karlsruhe Institute of Technology (KIT), Department of Mathematics, D-76128 Karlsruhe, Germany,\\
julia.hoerrmann@kit.edu}
\footnotetext[2]{Karlsruhe Institute of Technology (KIT), Department of Mathematics, D-76128 Karlsruhe, Germany, daniel.hug@kit.edu}

\maketitle 
\begin{abstract}
\noindent
We study a parametric class of isotropic but not necessarily stationary Poisson hyperplane tessellations in $n$-dimensional 
Euclidean space. Our focus is on the volume of the zero cell, i.e.\ the cell containing the origin. As a main result, we obtain an explicit formula for the variance of the volume of the zero cell in arbitrary dimensions. From this formula we deduce the asymptotic behaviour of the volume of the zero cell as the dimension 
goes to infinity.



\end{abstract}

\section{Introduction}
The majority of contributions to random tessellations is devoted to investigations 
in low and fixed dimensions. In particular, there exist only a few results on random tessellations in high dimensions, that is, with focus on asymptotic aspects as the dimension goes to infinity. 
Recently,  the typical cell of a stationary Poisson-Voronoi tessellation in high dimensions has been 
studied in \cite{AS}, \cite{Muche1} and \cite{Yao1}.
Alishahi and Sharifitabar \cite{AS} investigate the asymptotic behaviour of the volume and the shape of the  typical cell of a stationary Poisson-Voronoi tessellation as  the dimension $n$ of the space goes to infinity. In particular, 
they showed that the variance of the volume of the typical cell converges to zero exponentially fast as $n\to\infty$ whereas it is well  known that the expected volume is independent of the dimension. 
In the course of their investigation,  they made use of an explicit formula for the variance of the volume of the typical cell in arbitrary dimensions. 
The asymptotic behaviour of the volume of the typical cell  was studied earlier in the more general context of the nearest neighbour analysis by Newman et al. in \cite{Newman1983} and \cite{Newman1985}. In \cite[Theorem 10]{Newman1983}, they showed that, if the intensity of the underlying Poisson process is $\gamma$, then the $k$th moment of the volume of the typical cell converges to $\gamma^{-k}$ as the dimension goes to infinity and therefore, in particular, the volume of the typical cell converges in distribution to $\gamma^{-1}$.  
\medskip

In this work, we consider a parametric class of Poisson hyperplane tessellations 
and focus on the volume of the cell containing the origin (the zero cell). It is then natural to explore whether an asymptotic behaviour similar to that of the typical cell 
of a stationary Poisson-Voronoi tessellation is exhibited by the zero cell in the present more general class of random tessellations. 
An interesting family of not necessarily stationary or isotropic Poisson hyperplane tessellations is introduced in connection with the investigation of Kendall's conjecture in \cite{Hug3}. In the isotropic case, these Poisson hyperplane tessellations are completely determined by two parameters, the intensity $\gamma \in (0,\infty)$ and the distance exponent $r \in (0,\infty)$. 
For a special choice of the distance exponent and the intensity, the zero cell is equal in distribution to the typical cell of a stationary Poisson-Voronoi tessellation (cf.\ \cite{Hug3}, 
\cite[Sections 5.2 and 5.3.2]{Calka2010}). 
Therefore this family of isotropic Poisson hyperplane tessellations provides a general framework for investigating tessellations in high dimensions. 
Trying to extend the approach of Alishahi and Sharifitabar \cite{AS} to the wider context of the zero cell of this family of Poisson hyperplane tessellations, we came across the need for a formula for the variance of the volume of the zero cell. 
Finding a manageable expression turned out to be a rather complex issue. In fact, the formulas presented here mark the starting point of a more detailed study of the asymptotic behaviour of characteristics of the zero cell as  the dimension goes to infinity. Results concerning lower dimensional sections of the zero cell 
and other shape characteristics as well as a connection to the hyperplane conjecture will be considered separately.

\medskip

In the following, we give a more detailed overview of our results. A precise description of 
the particular (parametric) model of a Poisson hyperplane tessellation used here is given in Section 2. 
For this model we then derive, in Section 3,  an explicit expression for the expectation and bounds for the
moments of the volume of the zero cell in Proposition \ref{3.3}. 
An explicit expression for the second moment and the variance of the volume
of the zero cell is provided in Theorem \ref{2}. These results follow from a 
sequence of lemmas which make use of integral geometric transformations and  
 the symmetries of the geometric situation. 
In Theorem \ref{3.10} we deduce bounds for the variance of the volume of
the zero cell which involve auxiliary quantities $D(n,r)$ and $E(n,r)$. These
quantities have to be evaluated and bounded from above and below for a given distance
exponent $r$ and an intensity $\gamma$. 
In Corollary \ref{5.1} we consider the zero cell of a Poisson 
hyperplane tessellation with constant distance exponent $r$. The 
choice $r=1$ corresponds to a stationary Poisson hyperplane tessellation. For constant  
intensity we then prove that all moments as well as the variance of the volume
of the zero cell go to infinity as the dimension $n$ goes to
infinity.

In order to fix the expected volume of the zero cell, independent of the dimension, 
 the intensity of the underlying Poisson hyperplane process can 
 be chosen appropriately as a function of the  dimension $n$. However, 
 it follows from our bounds that as long as the distance exponent $r$ is fixed, 
 the variance of the volume of the zero cell still goes to infinity as  $n$ goes to infinity. The investigation in Section 4 thus suggests that in order 
to ensure that the variance converges to zero, the distance exponent $r$ 
has to be adjusted to the dimension $n$. 
In Corollary \ref{5.6} we summarize the case where the distance exponent $r$ is 
 proportional to the dimension $n$, i.e. $r=an$ with a fixed factor $a\in(0,\infty)$.
For constant intensity we show that the expectation and the moments of
the volume of the zero cell now all converge to zero as the dimension $n$ goes to
infinity.

In Theorem \ref{5.10} we finally consider the situation where 
 the distance exponent $r$ is proportional to the
dimension, i.e.\ $r=an$ with some fixed $a>0$, and the intensity 
$\widehat{\gamma}(a,n)$ is chosen as a function of the
dimension $n$ and the  factor $a$ in such a way that the
expected volume of the zero cell is equal to a positive constant. 
In this case we prove that the variance of the volume of the zero cell
converges to zero at an exponential speed of optimal order as $n\to\infty$. 
In particular, the volume of the zero cell converges in
distribution. 
In the special case $r=n$ (i.e.\ $a=1$), we fully recover results for the typical 
cell of a Poisson-Voronoi tessellation  obtained in \cite{AS}. 
The present more general approach applies to a larger class of tessellations and 
admits various other extensions and variations that will be discussed in detail in subsequent work.  
\section{Preliminaries}

In the following, we mainly use the notation and terminology of the monograph \cite{Sto}. 
We work in an $n$-dimensional real Euclidean vector space $\mathbb{R}^n$, $n\ge 2$, with scalar product
$\langle \cdot,\cdot \rangle$ and norm $\|\cdot\|$.
The unit ball $\{x \in \mathbb{R}^n: \|x\| \leq 1\}$ centred at the origin $o$ is denoted by $B^n$, its boundary is the unit sphere $S^{n-1}$.
For $k \in\{0,\ldots, n\}$, the Grassmannian of $k$-dimensional linear subspaces of $\mathbb{R}^n$ is denoted by $G(n,k)$, and the affine Grassmannian of $k$-dimensional affine subspaces ($k$-flats) by $A(n,k)$; both are equipped with their standard topologies.
For $u \in S^{n-1}$ and $t \in [0,\infty)$, we write
\[
H(u,t) := \{x \in \mathbb{R}^n : \langle x,u\rangle = t\}, \quad
H^{-}(u,t) := \{x \in \mathbb{R}^n : \langle x,u\rangle \leq t\}.
\]
Lebesgue measure on $\mathbb{R}^n$ is denoted by $\lambda$.
For $E\in G(n,k)$, Lebesgue measure on $E$ is denoted by $\lambda_E$. Besides we define
$S^{k-1}_E := E\cap S^{n-1}$ and $H_E(u,t):=E\cap H(u,t) $,  for $u \in S^{k-1}_E$ and 
$k\in\{1,\ldots,n\}$. 
The $s$-dimensional Hausdorff measure is denoted by
$\mathcal{H}^s$, where $s\ge 0$. For $s=n$ we sometimes refer to it as $n$-dimensional volume $V_n$.
A frequently occurring constant is the volume of the unit ball,
\[
\kappa_k := \lambda_k(B^k) = \frac{\pi^{\frac{k}{2}}}{\Gamma(\frac{k}{2}+1)},
\]
for $k\in\mathbb{N}_0$. 
The surface area of the unit sphere $S^{k-1}$, for $k\in\mathbb{N}$, is given by
\[
\omega_k := \mathcal{H}^{k-1}(S^{k-1})= k\kappa_k = \frac{2\pi^{\frac{k}{2}}}{\Gamma(\frac{k}{2})}.
\]
We repeatedly use that
\begin{equation}\label{eqGamma}
\int\limits_0^{\frac{\pi}{2}}(\sin\varphi)^\alpha (\cos\varphi)^\beta d\varphi = \frac{1}{2}\frac{\Gamma(\frac{\alpha+1}{2})\Gamma(\frac{\beta+1}{2})}{\Gamma(\frac{\alpha+\beta+2}{2})},
\end{equation}
for $\alpha,\beta>-1$; see \cite[(5.6)]{Artin1} or \cite[(12.42)]{WW}. In the following, we often use the connection 
between the Beta and the Gamma function (see \cite[(12.41)]{WW}) and Stirling's formula for the Gamma function 
(see \cite[(12.33)]{WW} or \cite[p.~24]{Artin1}). 
For $m\in \mathbb{N}$ and  $x_1,\ldots,x_m \in \R^n$, we denote by $[x_1,\dots,x_m]$ the convex hull and by $\spa\{x_1,\ldots,x_m\}$ the linear hull of $x_1,\ldots,x_m$.

The family of nonempty, compact, convex subsets of $\mathbb{R}^n$ is denoted by $\mathcal{K}^n$.
For a topological space $(T,\mathcal{T})$, a measure is always defined on the $\sigma$-algebra $\mathcal{B}({T})$ of Borel sets of ${T}$, i.e.\ the smallest $\sigma$-algebra containing the open sets $\mathcal{T}$.
We write $\mu^k := \mu \otimes \cdots \otimes \mu$, with $k$ factors $\mu$,  
for the $k$-fold product of a measure $\mu$.
By $SO_n$ we denote the group of proper (i.e.~rotation preserving) rotations on $\mathbb{R}^n$, and $\nu_n$  
is the unique Haar probability measure on $SO_n$.

The following setting has previously been considered in a more general, not necessarily isotropic framework,  
in the context of 
Kendall's problem \cite{Hug3} (see also \cite{HilhorstCalka08}, \cite{Calka2010}). Let 
$
(\Omega,\mathcal{A},\mathbb{P})$ denote the underlying probability space.  Further, 
let $X$ be a Poisson hyperplane process in $\mathbb{R}^n$, i.e.\ a Poisson point process in the space $A(n,n-1)$.
Subsequently, we identify a simple counting measure with its support, so that for a Borel set $A \subset A(n,n-1)$ both notations $X(A)$ and $\card(X \cap A)$ denote the number of elements of $X$ in A. We assume that the intensity measure
$\Theta(\cdot) = \mathbb{E}X(\cdot)$ of $X$ is of the form
\begin{equation}\label{intmeas}
\Theta (\cdot) = \frac{2\gamma}{n\kappa_n} \int\limits_{S^{n-1}}\int\limits_0^\infty 
\mathbf{1}\{H(u,t) \in \cdot \} t^{r-1}\,dt\,\mathcal{H}^{n-1}(du)
\end{equation}
with $\gamma >0$ and $r \in (0,\infty)$. We refer to $\gamma$ as the {\itshape intensity} and to $r$ as the 
{\itshape distance exponent} of the hyperplane process $X$.
Clearly, $\Theta$ is rotation invariant for all $r>0$. Furthermore, $\Theta$ is translation invariant only for $r=1$. 
Observe  that a calculation similar to the one required for \eqref{eq_c_nr} implies that $\Theta(\{H\in A(n,n-1): H\cap [-e,e]\neq \emptyset \}) = \Theta(\{H\in A(n,n-1): H\cap [0,2e]\})$ for $e\in S^{n-1}$ if and only if $r=1$. (It is well known that the intensity measure is translation invariant for $r=1$; cf.~ \cite[(4.33)]{Sto} and the reflection symmetry of $\mathcal{H}^{n-1}$ on the unit sphere.) 
 Therefore, since $X$ is a Poisson process, $X$ is always isotropic but stationary only for $r=1$.

The random polytope
\[
Z_0 := \bigcap\limits_{H \in X} H^{-},
\]
is the {\itshape zero cell} of the hyperplane process $X$, where $H^-$ 
denotes the (almost surely uniquely  determined) closed half-space bounded by $H$ which contains the origin. 
Clearly, $Z_0$ depends on $\gamma$ and $r$; although this dependence is not made explicit by our notation. 

For the distance exponent $r=n$ the zero cell $Z_0$ is equal in distribution to the typical cell of a stationary Poisson-Voronoi tessellation (see \cite{Hug3}). 
More detailed information on the topic of random tessellations is provided, e.g., in  \cite{Okabe1}, \cite{Sto}, 
\cite[Chapter 6]{Spodarev} and \cite{Stoyan1}.
Poisson-Voronoi tessellations have been studied extensively in the literature; see, for instance,  \cite{Hug2}, \cite{Moller1} and \cite{Muche2}.
%
%
Stationary Poisson hyperplane tessellations have been considered, e.g., in \cite{Hug4}, \cite{Hug7, Mecke5, Mecke6}.
Recently, also non-stationary Poisson hyperplane tessellations have attracted some attention (cf.\   \cite{Hug3}, 
 \cite{Schneider1}). 
A review of recent results on random polytopes is given in \cite{hug:Barany2008}, \cite{hug:Reitzner2010},  \cite{RecentResults},  \cite[Chapter 8]{Sto}, \cite[Chapter 7]{Spodarev}.

\section{A general formula for the variance}

Let $X$ be a Poisson hyperplane process in $\mathbb{R}^{n}$ with an intensity measure of the form
\eqref{intmeas}. We assume that $\gamma > 0$ and $r \in (0,\infty)$.
By Fubini's theorem and basic properties of a Poisson process, we have
\begin{align}
\mathbb{E}[V_n(Z_0)^k]  =& \mathbb{E} \Big[\int\limits_{\mathbb{R}^n}\mathbf{1}_{Z_0}(x_1)\,dx_1\cdots
\int\limits_{\mathbb{R}^n}\mathbf{1}_{Z_0}(x_k)\,dx_k \Big]\nonumber\\
 =& \int\limits_{(\mathbb{R}^n)^k}\mathbb{P}(x_1,\ldots,x_k\in Z_0)\,dx_1 \ldots dx_k \nonumber\\
 =& \int\limits_{(\mathbb{R}^n)^k}\exp \Big[-\frac{2\gamma}{n\kappa_n} \int\limits_{S^{n-1}}\int\limits_0^\infty\mathbf{1}\{H(u,t)\cap [o,x_1,\ldots,x_k] \neq \emptyset\}\nonumber\\
& \qquad \qquad\times t^{r-1}dt\,\mathcal{H}^{n-1}(du)\Big]\, dx_1\ldots dx_k.\label{kthmoment}
\end{align}
From \eqref{kthmoment} we now deduce lower and upper bounds for the moments of $V_n(Z_0)$.
For $e\in S^{n-1}$, we define 
\begin{equation}\label{eq_c_nr}
c(n,r) := \int\limits_{S^{n-1}}\langle e,u\rangle_+^r\,\mathcal{H}^{n-1}(du)
= \pi^{\frac{n-1}{2}}\frac{\Gamma(\frac{r+1}{2})}{\Gamma(\frac{r+n}{2})},
\end{equation}
which is indeed independent of the choice of the unit vector $e$. The explicit value 
is determined by a suitable decomposition of spherical Lebesgue measure (see \cite[(1.41)]{mueller}) 
and by an application of \eqref{eqGamma}.  
%


The following result provides bounds from above and below for the moments 
of the volume of the zero cell. Note that the ratio of 
the upper and the lower bound is given by the ratio of the corresponding values of the Gamma functions 
in these bounds. 

\begin{proposition}\label{3.3}
For $k\in\mathbb{N}$, we have
\[
\Gamma\left(\frac{n}{r}+1\right)^k \kappa_n^k\left(\frac{n\kappa_n r}{2\gamma c(n,r)}\right)^{\frac{kn}{r}}
\leq
\mathbb{E}[V_n(Z_0)^k]
\leq
\Gamma\left(\frac{kn}{r}+1\right)\kappa_n^k\left(\frac{n\kappa_n r}{2\gamma c(n,r)}\right)^{\frac{kn}{r}}.
\]
In particular, for $k=1$ we get 
\[
\mathbb{E}[V_n(Z_0)]=\Gamma\left(\frac{n}{r}+1\right)\kappa_n\left(\frac{n\kappa_n r}{2\gamma c(n,r)}\right)^{\frac{n}{r}}.
\]
\end{proposition}
\begin{proof}
Starting with \eqref{kthmoment}, introducing polar coordinates and by symmetry, we get
\begin{align*}
&\mathbb{E}[V_n(Z_0)^k]\\
& = k\int\limits_0^\infty\int\limits_0^{s_1}\ldots \int\limits_0^{s_1}\int\limits_{(S^{n-1})^k}
\exp\Big[-\frac{2\gamma}{n\kappa_n}\int\limits_{S^{n-1}}\int\limits_0^\infty\mathbf{1} \{H(u,t)\cap [o,s_1 v_1,\ldots,s_k v_k]\neq \emptyset\}\\
& \qquad\qquad \times t^{r-1}dt \,\mathcal{H}^{n-1}(du)\Big] s_1^{n-1}\ldots s_k^{n-1}\, \mathcal{H}^{n-1}(dv_k)\ldots \mathcal{H}^{n-1}(dv_1)\,ds_k\ldots ds_1\\
&\leq  k\int\limits_0^\infty\int\limits_0^{s_1}\ldots\int\limits_0^{s_1}\int\limits_{(S^{n-1})^k}
\exp\Big[-\frac{2\gamma}{n\kappa_n}\int\limits_{S^{n-1}}\int\limits_0^\infty \mathbf{1} \{H(u,t)\cap [o,s_1 v_1]\neq \emptyset\}\\
& \qquad\qquad \times t^{r-1}\,dt\,\mathcal{H}^{n-1}(du)\Big] s_1^{n-1}\ldots s_k^{n-1}\mathcal{H}^{n-1}(dv_k)\ldots \mathcal{H}^{n-1}(dv_1)\,ds_k\ldots ds_1.
\end{align*}
Since
\begin{align*}
&\int\limits_{S^{n-1}}\int\limits_0^\infty \mathbf{1} \{H(u,t)\cap [o,s_1 v_1]\neq \emptyset\}t^{r-1}\,dt\,\mathcal{H}^{n-1}(du)\\
&=\int\limits_{S^{n-1}}\int\limits_0^\infty \mathbf{1} \{0 \leq t \leq  \langle s_1 v_1, u\rangle_+\}t^{r-1}\,dt\,\mathcal{H}^{n-1}(du) \\
&=\frac{1}{r} s_1^r c(n,r),
\end{align*}
the upper bound follows easily. 

For the lower bound, we observe that
$H(u,t)\cap [o,x_1,\cdots,x_k] \neq \emptyset$ if and only if there is some  $ i \in \{1,\cdots,k\}$ 
such that $ H(u,t) \cap [o,x_i] \neq \emptyset$. 
Hence, $\mathbb{E}[V_n(Z_0)^k]$ can be bounded from below by
\begin{align*}
&\int\limits_{(\mathbb{R}^n)^k}
\exp\Big[-\frac{2\gamma}{n\kappa_n}\int\limits_{S^{n-1}}\int\limits_0^\infty \sum\limits_{i=1}^k\mathbf{1} \{H(u,t)\cap [o,x_i]\neq \emptyset\}t^{r-1}\,d t\,\mathcal{H}^{n-1}(d u)\Big] \, d x_1\ldots d x_k\\
&=\left(\int_{\R^n}\exp\Big[-\frac{2\gamma\|x_1\|^r}{n\kappa_n r}c(n,r)\Big]\, dx_1\right)^k
.
\end{align*}
Now the assertion can be shown by a straightforward calculation.
\end{proof}

\bigskip

\begin{remark}\label{3.4}
 The lower bound in Proposition \ref{3.3} can also be obtained by an application of H\"older's inequality. 
 Moreover, for fixed $r$ and $k$, the ratio of the upper and the lower bound is of the order 
$n^{\frac{1-k}{2}}k^{\frac{kn}{r}}.$
 
\end{remark}
%

\medskip

For the statement of Theorem \ref{2}, which provides formulae for  the second moment and the variance of $V_n(Z_0)$, we need the constant
\[b_{n,2}:= \frac{\omega_{n-1}\omega_{n}}{4 \pi}\]
and the auxiliary functions
\[\alpha(t,\varphi) := \arctan\left(\frac{t-\cos\varphi}{\sin\varphi}\right)
\in\left(-\frac{\pi}{2},\frac{\pi}{2}\right)\]
and
\[
F_r(t,\varphi):= \frac{1}{c(2,r)} \left( t^r\,\int\limits_{-\frac{\pi}{2}}^{\alpha(t,\varphi)}(\cos\theta)^r\, d\theta+\int\limits_{\alpha(t,\varphi)-\varphi}^{\frac{\pi}{2}}
(\cos\theta)^r\, d\theta \right) 
\]
for $(t,\varphi)\in [0,1]\times(0,\pi)$. 
In the following, we will use that
\begin{equation}\label{Fbound}
\frac{1}{2}\le F_r(t,\varphi)\leq t^r+1\le  2 
\end{equation}
for all $(t,\varphi)\in [0,1]\times(0,\pi)$, 
which follows from the subsequent Remark \ref{3.9}, (1) and (2).

\begin{theorem}\label{2}
Let $X$ be a Poisson hyperplane process in $\mathbb{R}^{n}$ with an intensity measure of the form \eqref{intmeas} with  intensity $\gamma > 0$ and distance exponent $r \in (0,\infty)$.
Then
\begin{align*}
\mathbb{E}[V_n(Z_0)^2]
=& \,\frac{8\pi b_{n,2}}{r} \,\Gamma\left(\frac{2n}{r}\right) \left(\frac{n\kappa_n r}{2\gamma c(n,r)}\right)^{\frac{2n}{r}}\int\limits_0^\pi \int\limits_0^1 \frac{t^{n-1}}{F_r(t,\varphi)^{\frac{2n}{r}}} (\sin \varphi)^{n-2}\,dt\,d\varphi
\end{align*}
and
\begin{align*}
 \Var[V_n(Z_0)]
 = &\, \frac{8\pi b_{n,2}}{r}\, \Gamma\left(\frac{2n}{r}\right) \left(\frac{n\kappa_n r}{2\gamma c(n,r)}\right)^{\frac{2n}{r}}\\
&\, \times \int\limits_0^\pi \int\limits_0^1 \left(\frac{1}{F_r(t,\varphi)^{\frac{2n}{r}}} - \frac{1}{(t^r+1)^{\frac{2n}{r}}}\right) t^{n-1}(\sin \varphi)^{n-2}\,dt\,d\varphi.
\end{align*}
\end{theorem}
\begin{proof}
The formula for $\mathbb{E}[V_n(Z_0)^2]$ is implied  by the subsequent Lemma \ref{3.7}, Lemma \ref{3.8} and 
straightforward calculations. 
 The formula for $\Var[V_n(Z_0)]$ then follows if we also use the representation of $\big(\mathbb{E}[V_n(Z_0)]\big)^2$ as a double integral, i.e.
\[
\big(\mathbb{E}[V_n(Z_0)]\big)^2
%
= \frac{8\pi b_{n,2}}{r} \,\Gamma\left(\frac{2n}{r}\right) \left(\frac{n\kappa_n r}{2\gamma c(n,r)}\right)^{\frac{2n}{r}}\int\limits_0^\pi \int\limits_0^1   \frac{t^{n-1}}{(t^r+1)^{\frac{2n}{r}}}(\sin \varphi)^{n-2}\,dt\,d\varphi,
\]
which is a consequence of the special case $k=1$ of Proposition \ref{3.3} and of the fact that
\[
\int\limits_0^1\frac{t^{n-1}}{(t^r+1)^{\frac{2n}{r}}}\, dt = \frac{1}{2r} \frac{\Gamma(\frac{n}{r})^2}{\Gamma(\frac{2n}{r})},
\]
which follows from the substitution $z=(t^r+1)^{-1}$, the symmetry of the resulting integrand with respect to $1/2$ and basic calculations involving Beta and Gamma functions. Furthermore, we use that
\[
n^2 \kappa_n^2 = 4\pi b_{n,2} \int\limits_0^\pi(\sin\varphi)^{n-2}\, d\varphi,
\]
which follows from the definition of $\kappa_n$ and $b_{n,2}$ and from a special case of \eqref{eqGamma}.
\end{proof}

\bigskip
\begin{figure}
\centering
\includegraphics[trim=19 0 4 0, width=400pt, height=200pt, clip=true]{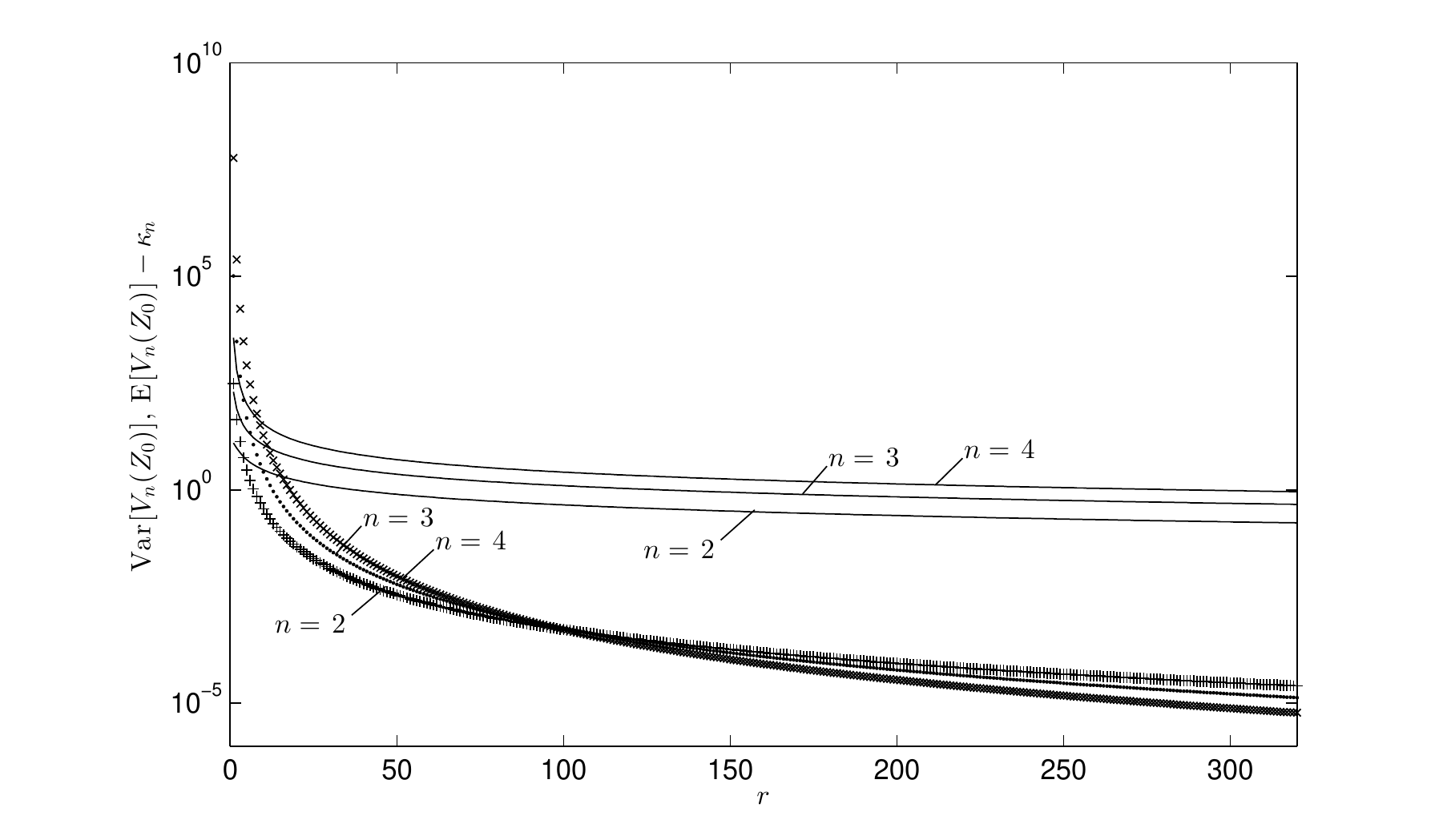}
\includegraphics[width=400pt, height=200pt]{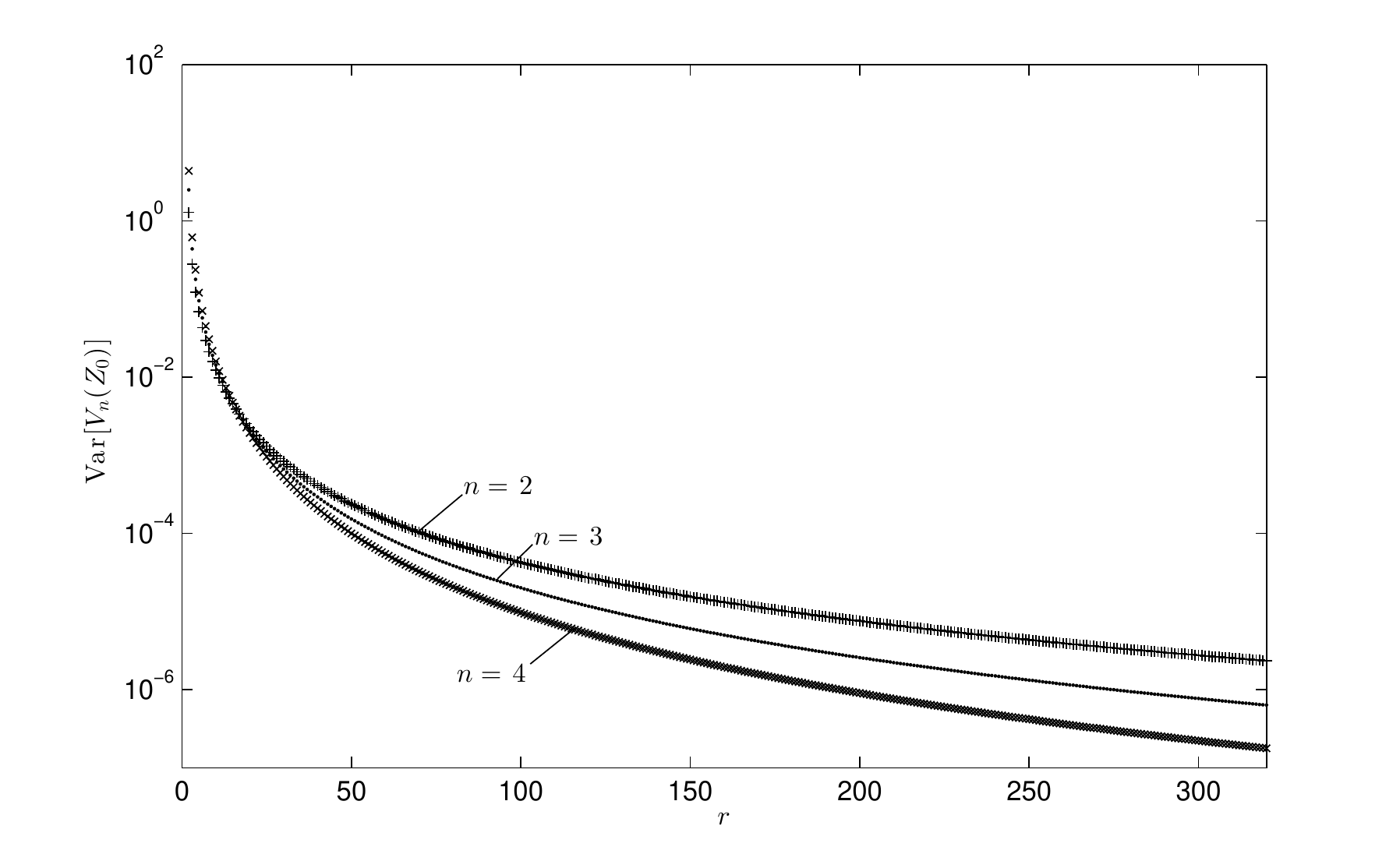}
\caption{Numerical evaluation of the formula for the variance from Theorem \ref{2} using the numerical integration functions of MATHEMATICA$^\text{\textregistered}$. In the top panel, we fix $\gamma=1$ and $\mathbb{E}[V_n(Z_0)]-\lim\limits_{r\rightarrow \infty}\mathbb{E}[V_n(Z_0)]=\mathbb{E}[V_n(Z_0)]-\kappa_n$ is plotted as a solid line for comparison. In the bottom panel, $\gamma$ is chosen in such a way 
that $\mathbb{E}[V_n(Z_0)]=1$.}\label{f1}
\end{figure}
\begin{figure}
\centering
\includegraphics[width=400pt, height=200pt]{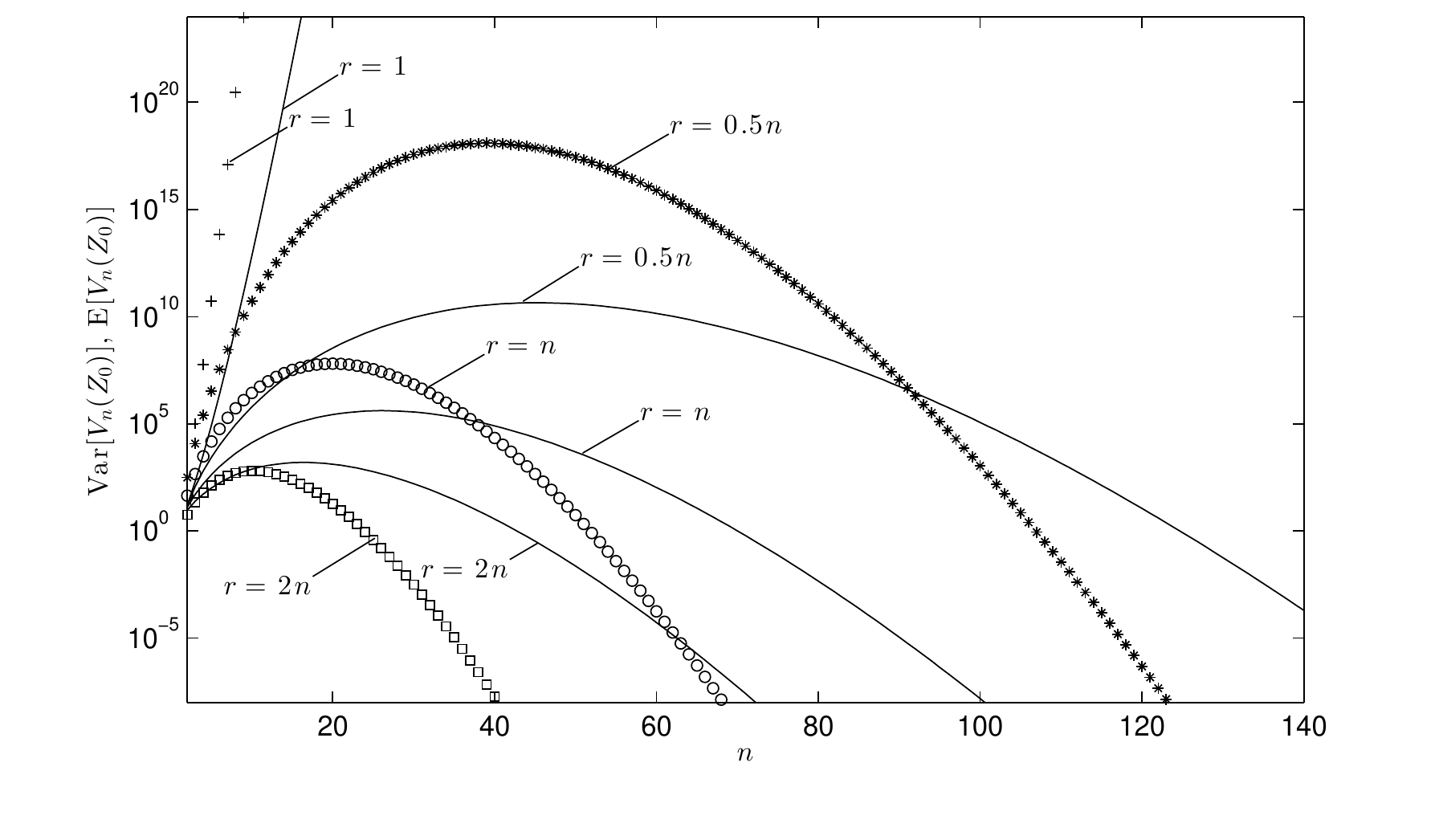}
\includegraphics[trim= 0 -15 0 15, width=400pt, height=200pt]{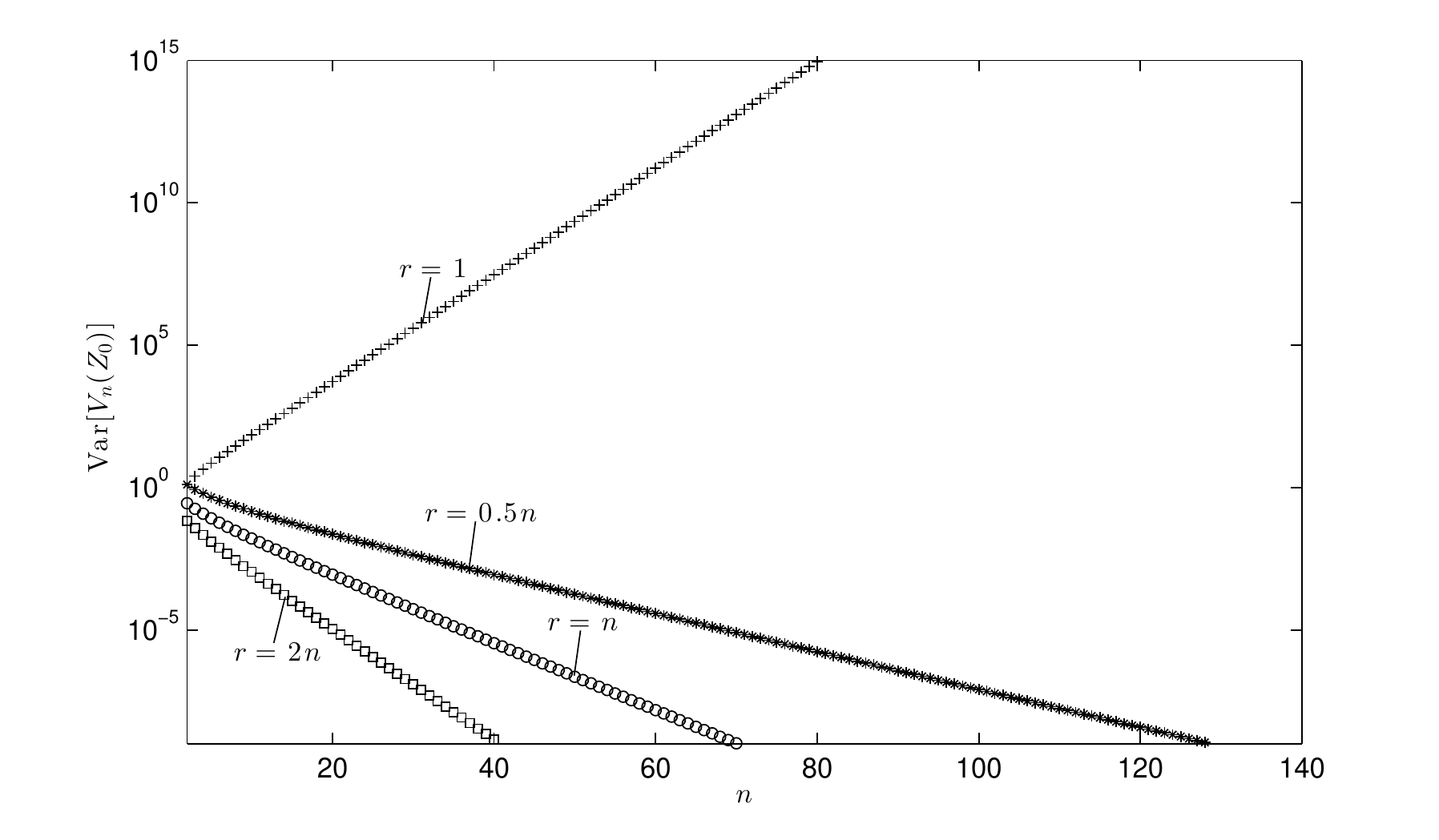}
\caption{Numerical evaluation of the formula for the variance from Theorem \ref{2} using the numerical integration functions of MATHEMATICA$^\text{\textregistered}$. In the top panel, we fix $\gamma=1$ and $\mathbb{E}[V_n(Z_0)]$ is plotted as a solid line for comparison. In the bottom panel, $\gamma$ is chosen in such a way that $\mathbb{E}[V_n(Z_0)]=1$.}\label{f2}
\end{figure}

\begin{remark}
In Figure \ref{f1} and Figure \ref{f2} numerical calculations of the variance based on Theorem \ref{2} are added.
In small dimensions ($n=2,3,4$) the numerical calculations of the variance for varying $r$  are plotted for $\gamma=1$. In this case observe that by Proposition \ref{3.3} and Stirling's formula we have $\lim\limits_{r\rightarrow \infty} \mathbb{E}[V_n(Z_0)^k]=\kappa_n^k$ for 
$k\in\mathbb{N}$ and arbitrary fixed $n\geq 2$. 
In addition we study the choice $\gamma=\frac{n\kappa_n r}{2 c(n,r)}\left(\Gamma(\frac{n}{r}+1)\kappa_n\right)^{\frac{r}{n}}$, which implies $\mathbb{E}[V_n(Z_0)]=1$ and $\lim\limits_{r\rightarrow \infty} \mathbb{E}[V_n(Z_0)^k]=1$ for $k\in\mathbb{N}$ and arbitrary fixed $n \geq 2$ by Proposition \ref{3.3}.
Theorem \ref{3.10} and Lemma \ref{3.12}, which will be proved in the following section, and Stirling's formula show that for both choices of $\gamma$ we obtain 
\[\Var[V_n(Z_0)] \leq C\, r^{-\frac{n+1}{2}},\]
where $C>0$ is a constant depending on $\gamma$ and $n$ in the first case and only on $n$ in the second. 
On the other hand, for specific choices of $r$ ($r=1$, $r=0.5\,n$, $r=n$, $r=2n$) the numerically determined values of the variance for varying dimension $n$ are plotted for $\gamma=1$ and $\gamma=\frac{n\kappa_n r}{2 c(n,r)}\left(\Gamma(\frac{n}{r}+1)\kappa_n\right)^{\frac{r}{n}}$. The high-dimensional limiting behaviour of the moments and the variance is studied in the following section.
\end{remark}

In the following remark, we collect some facts which are helpful for a proper understanding of the formulas for the second moment and the variance of the volume of the zero cell in Theorem \ref{2} and which are used several times subsequently.

\begin{remark}\label{3.9}
\begin{enumerate}
\item[(1)] Let $\varphi \in (0,\pi)$ and $t\in[0,1]$.  Then
\[
-\frac{\pi}{2} < \varphi-\frac{\pi}{2}=\alpha(0,\varphi) \leq \alpha(t,\varphi) \leq \alpha(1,\varphi) = \frac{\varphi}{2}<\frac{\pi}{2},
\]
since $t\mapsto \alpha(t,\varphi)$ is  increasing on $[0,1]$.
\item[(2)] A special case of \eqref{eqGamma} is
\[
\int\limits_{-\frac{\pi}{2}}^{\frac{\pi}{2}}(\cos\theta)^r \,d\theta = c(2,r).
\]
%
\end{enumerate}
\end{remark}

\bigskip

The following lemmas lead successively to the explicit formulas for the second moment and the variance of the volume of the zero cell stated in Theorem \ref{2}. Lemma \ref{3.5} and Lemma \ref{3.6} are needed to prove Lemma \ref{3.7}, whereas Lemma \ref{3.7} and Lemma \ref{3.8} have been used directly in the proof of Theorem \ref{2}.
In a first step,  the integral representation of the second moment of $V_n(Z_0)$ from \eqref{kthmoment} will be simplified considerably by an application of a Blaschke-Petkantschin formula (cf.\  \cite{Sto}).
For  $x_1,x_2 \in \mathbb{R}^n$, let $\nabla_2(x_1,x_2)$ denote the area of the parallelogram spanned by these vectors. Further, denote by 
$\nu_2^n$ the unique rotation invariant  probability measure on $G(n,2)$. As usual, $\{e_1,\ldots,e_n\}$ denotes the standard basis of $\mathbb{R}^n$. 
\begin{lemma}\label{3.5}
We have
\begin{eqnarray*}
\mathbb{E}[V_n(Z_0)^2] &=& 
b_{n,2} \int\limits_{\spa\{e_1,e_2\}^2}\exp\Big[-\frac{2\gamma}{n\kappa_n}
\int\limits_{S^{n-1}}\int\limits_0^\infty\mathbf{1}
\{H(u,t)\cap [o,x_1,x_2]\neq \emptyset\}t^{r-1}\\
&&\qquad\qquad\qquad\times \,dt\,\mathcal{H}^{n-1}(du)\Big]\nabla_2(x_1,x_2)^{n-2}\,
\lambda^2_{\spa\{e_1,e_2\}}(d(x_1,x_2)).
\end{eqnarray*}
\end{lemma}
\begin{proof}
From \eqref{kthmoment}, the linear Blaschke-Petkantschin formula \cite[Theorem 7.2.1]{Sto}, 
the rotation invariance of spherical Lebesgue measure and the invariance of $\nabla_2(\cdot,\cdot)$, 
we obtain that
\begin{align*}
& \mathbb{E}[V_n(Z_0)^2] \\
& = b_{n,2}\int\limits_{G(n,2)} \int\limits_{L^2} \exp\Big[-\frac{2\gamma}{n\kappa_n} \int\limits_{S^{n-1}} \int\limits_0^\infty\mathbf{1}\{H(u,t)\cap
[o,x_1,x_2]\neq\emptyset\}\\ 
& \qquad\qquad  \times t^{r-1}\,dt\,\mathcal{H}^{n-1}(du)\Big]\nabla_2(x_1,x_2)^{n-2}\,\lambda_L^2(d(x_1,x_2))\,\nu_2^n(dL)\\
& = b_{n,2}\int\limits_{SO_n}\int\limits_{\spa\{e_1,e_2\}^2}\exp\Big[-\frac{2\gamma}{n\kappa_n} \int\limits_{S^{n-1}}\int\limits_0^\infty \mathbf{1} \{H(u,t) \cap
[o,\vartheta x_1, \vartheta x_2]\neq \emptyset\}\\
& \qquad\qquad \times t^{r-1}\,dt \,\mathcal{H}^{n-1}(du)\Big] \nabla_2(\vartheta x_1,\vartheta x_2)^{n-2}\, \lambda^2_{\spa\{e_1,e_2\}}(d(x_1,x_2)) \,\nu_n(d\vartheta)\\
& = b_{n,2}\int\limits_{SO_n}\int\limits_{\spa\{e_1,e_2\}^2}\exp\Big[-\frac{2\gamma}{n\kappa_n} \int\limits_{S^{n-1}}\int\limits_0^\infty \mathbf{1}\{ H(\vartheta u, t) \cap [o, \vartheta x_1, \vartheta x_2] \neq \emptyset\}\\
& \qquad\qquad \times t^{r-1}\,dt\, \mathcal{H}^{n-1}(du)\Big] \nabla_2( x_1,  x_2)^{n-2}\, \lambda^2_{\spa\{e_1,e_2\}}(d(x_1,x_2))\,\nu_n(d\vartheta)\\
& = b_{n,2}\int\limits_{\spa\{e_1,e_2\}^2}\exp\Big[-\frac{2\gamma}{n\kappa_n}
\int\limits_{S^{n-1}}\int\limits_0^\infty \mathbf{1} \{H(u,t) \cap
[o,x_1,x_2]\neq\emptyset\}\\ 
& \qquad\qquad \times t^{r-1}\,dt\, \mathcal{H}^{n-1}(du)\Big] \nabla_2(x_1,x_2)^{n-2}\, \lambda^2_{\spa\{e_1,e_2\}}(d(x_1,x_2)),
\end{align*}
which yields the assertion of the lemma. 
\end{proof}
\bigskip

Next we simplify the inner double integral of the expression which was derived in Lemma \ref{3.5}
for $\mathbb{E}[V_n(Z_0)^2]$ by exploiting further the symmetry of the situation.
\begin{lemma}\label{3.6}
For $x_1, x_2 \in \spa\{e_1,e_2\} \subset \mathbb{R}^n$, we have
\begin{align*}
& \int\limits_{S^{n-1}}\int\limits_0^\infty\mathbf{1}\{H(u,t)\cap [o,x_1,x_2]\neq \emptyset\}t^{r-1}\, dt\,\mathcal{H}^{n-1}(du)\\
& =
 \frac{c(n,r)}{c(2,r)} \int\limits_0^{2\pi}\int\limits_0^\infty\mathbf{1}\big\{H(\begin{pmatrix}\cos \theta \\
\sin \theta\end{pmatrix},t) \cap [o,x_1,x_2] \neq \emptyset\big\} t^{r-1}\, dt\, d\theta.
\end{align*}
\end{lemma}
\begin{proof} For $n=2$ there is nothing to prove. Hence we can assume that $n\ge 3$. 
Let $E:= \spa\{e_1,e_2\}$. The map
\[
F: \begin{cases}\begin{array}{r c l}
S^1_E\times (0,\frac{\pi}{2})\times S^{n-3}_{E^\bot} & \rightarrow &S^{n-1},\\
(u_1,\theta,u_2) & \mapsto& \cos(\theta) u_1 + \sin(\theta)u_2,
\end{array}
\end{cases}
\]
is injective and its image covers $S^{n-1}$ up to a set of measure zero.
Its Jacobian is
\[
J F(u_1,\theta,u_2) = \cos(\theta) (\sin(\theta))^{n-3},
\]
and hence the  area-coarea formula (cf.\ \cite[Theorem 3.2.22]{GM}) yields that
\begin{align*}
& \int\limits_{S^{n-1}}\int\limits_0^\infty\mathbf{1}\{H(u,t)\cap  [o,x_1,x_2]\neq \emptyset\}t^{r-1}\, dt\,\mathcal{H}^{n-1}(du)\\
& = \int\limits_{S^1_E}\int\limits_0^{\frac{\pi}{2}}\int\limits_{S^{n-3}_{E^\bot}}
\int\limits_0^\infty\mathbf{1}\{H(\cos(\theta)u_1+\sin(\theta)u_2,t)\cap [o,x_1,x_2]\neq \emptyset\} \\
& \qquad\qquad \times \cos(\theta) (\sin \theta )^{n-3}t^{r-1}\, dt\, \mathcal{H}^{n-3}(du_2)\,
d\theta \, \mathcal{H}^1(du_1).
\end{align*}
Since $[o,x_1,x_2] \subset E$ and
\[
H(\cos(\theta)u_1+\sin(\theta)u_2,t)\cap E = H_E\left(u_1,\frac{t}{\cos \theta }\right),
\]
for $\theta\in [0,\pi/2)$, we get
\begin{align*}
& \int\limits_{S^{n-1}}\int\limits_0^\infty\mathbf{1}\{H(u,t)\cap  [o,x_1,x_2]\neq \emptyset\}t^{r-1}\, dt\,\mathcal{H}^{n-1}(du)\\
& = (n-2)\kappa_{n-2} \int\limits_{S^1_E}\int\limits_0^{\frac{\pi}{2}}
\int\limits_0^\infty\mathbf{1}\{H_E\left(u_1,\frac{t}{\cos \theta }\right)\cap [o,x_1,x_2]\neq \emptyset\}\\
& \qquad\qquad \times \cos(\theta) (\sin\theta)^{n-3} t^{r-1}\, dt\,
d\theta \, \mathcal{H}^1(du_1)\\
& = (n-2)\kappa_{n-2}\int\limits_{S^1_E}\int\limits_0^{\frac{\pi}{2}}
\int\limits_0^\infty\mathbf{1}\{H_E(u_1,t)\cap [o,x_1,x_2]\neq \emptyset\} \\
& \qquad\qquad \times (\cos\theta)^{r+1} (\sin\theta)^{n-3} t^{r-1}\, dt\,
d\theta \, \mathcal{H}^1(du_1)\\
& = (n-2)\kappa_{n-2} \int\limits_0^{\frac{\pi}{2}}(\cos\theta)^{r+1}(\sin\theta)^{n-3}d\theta\, \\
& \qquad\qquad \times
\int\limits_{S^1}\int\limits_0^\infty\mathbf{1}\{H(u_1,t)\cap [o,x_1,x_2]\neq \emptyset\} t^{r-1}\,dt\,\mathcal{H}^1(du_1)\\
& = \frac{c(n,r)}{c(2,r)}
\int\limits_0^{2\pi}\int\limits_0^\infty\mathbf{1}\big\{H\big(\begin{pmatrix}\cos\theta\\ \sin\theta\end{pmatrix},t\big) \cap [o,x_1,x_2] \neq \emptyset\big\} t^{r-1}\,
dt\, d\theta,
\end{align*}
which completes the proof of the lemma. 
\end{proof}

\bigskip

Having simplified the inner integral of the expression found 
in Lemma \ref{3.5} for $\mathbb{E}[V_n(Z_0)^2]$, we reduce the 
outer integral in the next lemma by again taking advantage of the problem's symmetry.

\begin{lemma}\label{3.7}
We have
\begin{align*}
\mathbb{E}[V_n(Z_0)^2] = &\, 8\pi b_{n,2} \int\limits_0^\pi \int\limits_0^\infty\int\limits_0^s \exp\Big[-\frac{2\gamma}{n\kappa_n} \frac{c(n,r)}{c(2,r)} \\
& \qquad\times\int\limits_0^{2\pi}\int\limits_0^\infty
\mathbf{1}\Big\{H\big(\begin{pmatrix}\cos\theta \\ \sin \theta \end{pmatrix},t\big)\cap
\big[o,s\begin{pmatrix}\cos\varphi\\ \sin\varphi\end{pmatrix},\begin{pmatrix}u\\ 0\end{pmatrix}\big]\neq \emptyset\Big\} t^{r-1}dt\, d\theta\Big]\\
& \qquad \times \,s^{n-1}u^{n-1} (\sin\varphi)^{n-2} \, du\, ds\, d\varphi.
\end{align*}
\end{lemma}
\begin{proof}
Combining Lemma \ref{3.5} and Lemma \ref{3.6} and introducing polar coordinates, we get
\begin{align*}
& \mathbb{E}[V_n(Z_0)^2] \\
& = b_{n,2} \int\limits_0^{2\pi}\int\limits_0^{2\pi}\int\limits_0^\infty\int\limits_0^\infty
\exp\Big[-\frac{2\gamma}{n\kappa_n} \frac{c(n,r)}{c(2,r)}\\
& \qquad \times \int\limits_0^{2\pi}\int\limits_0^\infty \mathbf{1}\Big\{H\big(\begin{pmatrix}\cos\theta\\ \sin\theta\end{pmatrix},t\big) \cap
\big[o, s\begin{pmatrix}\cos\varphi\\ \sin\varphi\end{pmatrix}, u\begin{pmatrix}\cos\psi\\ \sin\psi\end{pmatrix}\big] \neq\emptyset\Big\} t^{r-1}\,dt\, d\theta\Big]\\
& \qquad \times u s |\sin(\varphi-\psi)us|^{n-2}\,du\, ds\, d\varphi\, d\psi\\
& = b_{n,2} \int\limits_0^{2\pi}\int\limits_0^{2\pi}2\int\limits_0^\infty\int\limits_0^s
\exp\Big[-\frac{2\gamma}{n\kappa_n} \frac{c(n,r)}{c(2,r)}\\
& \qquad \times \int\limits_0^{2\pi}\int\limits_0^\infty \mathbf{1}\Big\{H\big(\begin{pmatrix}\cos\theta\\ \sin\theta\end{pmatrix},t\big) \cap
\big[o, s\begin{pmatrix}\cos\varphi\\ \sin\varphi\end{pmatrix}, u\begin{pmatrix}\cos\psi\\ \sin\psi\end{pmatrix}\big] \neq\emptyset\Big\} t^{r-1}\,dt\, d\theta\Big]\\
& \qquad \times |\sin(\varphi-\psi)|^{n-2} (us)^{n-1}\,du\, ds\, d\varphi\, d\psi,\\
\end{align*}
where the symmetry in $u$ and $s$ is used to justify the second equality.
Hence we derive
\begin{align*}
& \mathbb{E}[V_n(Z_0)^2]\\
& = 2b_{n,2}\int\limits_0^{2\pi}\int\limits_0^{2\pi}\int\limits_0^\infty\int\limits_0^s \exp\Big[-\frac{2\gamma }{n\kappa_n} \frac{c(n,r)}{c(2,r)}\\
&\qquad\times \int\limits_0^{2\pi}\int\limits_0^\infty \mathbf{1} \Big\{H\big(\begin{pmatrix}\cos(\theta-\psi)\\ \sin(\theta- \psi)\end{pmatrix},t\big) \cap \big[o, s \begin{pmatrix}\cos(\varphi-\psi)\\ \sin(\varphi-\psi)\end{pmatrix},
\begin{pmatrix}u\\ 0\end{pmatrix}\big] \neq \emptyset\Big\} t^{r-1}\,dt\, d\theta\Big]\\
& \qquad \times u^{n-1} s^{n-1} |\sin(\varphi-\psi)|^{n-2} \,du\, ds\, d\varphi\, d\psi\\
& = 4\pi b_{n,2} \int\limits_0^{2\pi} \int\limits_0^\infty \int\limits_0^s \exp\Big[-\frac{2\gamma }{n\kappa_n}
\frac{c(n,r)}{c(2,r)}\\
& \qquad \times \int\limits_0^{2\pi} \int\limits_0^\infty \mathbf{1} \Big\{ H\big(\begin{pmatrix}\cos\theta\\ \sin\theta\end{pmatrix},t\big) \cap \big[o,s \begin{pmatrix}\cos\varphi\\ \sin\varphi \end{pmatrix}, \begin{pmatrix} u\\ 0 \end{pmatrix}\big] \neq \emptyset\Big\} t^{r-1}\,dt\, d\theta\Big]\\
& \qquad \times u^{n-1} s^{n-1} |\sin\varphi|^{n-2} \,du\, ds\, d\varphi\\
& =8\pi b_{n,2}   \int\limits_0^{\pi} \int\limits_0^\infty \int\limits_0^s \exp\Big[-\frac{2\gamma }{n\kappa_n}
\frac{c(n,r)}{c(2,r)} 
\\
& \qquad \times \int\limits_0^{2\pi}\int\limits_0^\infty \mathbf{1}
\Big\{H\big(\begin{pmatrix}\cos\theta\\ \sin\theta\end{pmatrix},t\big) \cap \big[ o, s\begin{pmatrix}\cos\varphi\\  \sin\varphi\end{pmatrix}, \begin{pmatrix}u\\0 \end{pmatrix}\big] \neq \emptyset \Big\} t^{r-1} \,dt\, d\theta\Big]\\
& \qquad \times s^{n-1} u^{n-1}
(\sin\varphi)^{n-2} \,du\, ds\, d\varphi.
\end{align*}
For the last equation, we used that 

\begin{align*}
&   \int\limits_\pi^{2\pi} \int\limits_0^\infty \int\limits_0^s \exp\Big[-\frac{2\gamma }{n\kappa_n} \frac{c(n,r)}{c(2,r)}\\
& \qquad  \times \int\limits_0^{2\pi}\int\limits_0^\infty \mathbf{1}
\Big\{H\big(\begin{pmatrix}\cos\theta\\ \sin\theta\end{pmatrix},t\big) \cap \big[ o, s \begin{pmatrix} \cos\varphi\\ \sin\varphi \end{pmatrix}, \begin{pmatrix}u\\ 0\end{pmatrix}\big] \neq \emptyset \Big\} t^{r-1} \,dt\, d\theta\Big]\\
& \qquad \times u^{n-1} s^{n-1} |\sin\varphi|^{n-2} \,du\, ds\, d\varphi\\
& =  \int\limits_0^\pi \int\limits_0^\infty \int\limits_0^s
\exp\Big[-\frac{2\gamma}{n\kappa_n} \frac{c(n,r)}{c(2,r)}\\
& \qquad \times \int\limits_0^{2\pi} \int\limits_0^\infty \mathbf{1} \Big\{H\big(\begin{pmatrix}\cos\theta\\ \sin\theta \end{pmatrix},t\big) \cap \big[o, s \begin{pmatrix}\cos\varphi\\ \sin\varphi \end{pmatrix}, \begin{pmatrix}u\\ 0 \end{pmatrix}\big] \neq \emptyset\Big\} t^{r-1} \,dt\, d\theta\Big]\\
& \qquad \times u^{n-1} s^{n-1} (\sin\varphi)^{n-2} \,du\, ds\, d\varphi,
\end{align*}
which follows from the invariance of the inner integral under reflection of $\varphi$ and $\theta$ with respect to the first axis (that is, invariance with respect to replacing $\varphi$ by $-\varphi$ and $\theta$ by $-\theta$). This completes the proof. 
\end{proof}

\bigskip

In the expression found in Lemma \ref{3.7} the indicator function depends on $\theta,t,s,u$ and $\varphi$. 
Its support can be determined explicitly. This is used in the proof of the following lemma where it is shown 
how the integration with respect to $t$ in Lemma \ref{3.7} can be carried out.
\begin{lemma}\label{3.8}
For $ u ,s \in (0,\infty)$, $u\le s$ and $\varphi \in (0,\pi)$,
\[ \int\limits_0^{2\pi}\int\limits_0^\infty\mathbf{1}\Big\{H\big(\begin{pmatrix}\cos \theta\\ 
\sin \theta \end{pmatrix},t\big)\cap \big[o,s\begin{pmatrix}\cos \varphi\\ 
\sin \varphi \end{pmatrix},\begin{pmatrix}u\\ 0 \end{pmatrix}\big] \neq \emptyset \Big\}t^{r-1}\,dt\, d\theta =\frac{s^r c(2,r)}{r} F_r\left(\frac{u}{s},\varphi\right).
\]
\end{lemma}
\begin{proof}

The value of the integral on the left-hand side of the asserted equation does not change if we choose $[-\pi/2,3\pi/2]$ instead of $[0,2\pi]$ as the integration domain. Thus we have to determine the support of the indicator function under the integral on $[-\pi/2,3\pi/2]\times[0,\infty)$. Let $\varphi\in (0,\pi)$ be fixed. Then the indicator function is $1$ if and only if
\begin{equation}\label{seccases}
H\big(\begin{pmatrix}\cos\theta \\ \sin\theta\end{pmatrix},t\big) \cap
\big[o, s \begin{pmatrix}\cos\varphi\\ \sin\varphi\end{pmatrix}\big] \neq \emptyset \quad \text{or}\quad
H\big(\begin{pmatrix}\cos\theta \\ \sin\theta\end{pmatrix},t\big) \cap
\big[o,\begin{pmatrix}u\\ 0\end{pmatrix}\big] \neq \emptyset,
\end{equation}
which is satisfied if and only if
\[
t \in [0,u(\cos\theta)_+] \cup [0, s(\cos(\theta-\varphi))_+].
\]
If $\theta\in [-\pi/2,\varphi-\pi/2]$, this is equivalent to $t\in [0,u\cos\theta]$.
If $\theta\in [\varphi-\pi/2,\pi/2]$, this is equivalent to $t\in [0,\max\{u\cos\theta,s\cos(\theta-\varphi)\}]$.
If $\theta\in [\pi/2, \varphi+\pi/2]$, this is equivalent to $t\in [0,s\cos(\theta-\varphi)]$. For all other choices of $\theta\in [-\pi/2,3\pi/2]$, this is equivalent to $t=0$, and hence  can be neglected for the integration. Since
\begin{align*}
u\cos\theta = s \cos( \theta-\varphi) \; \Leftrightarrow \; \theta = \arctan\left(\frac{\frac{u}{s}-\cos\varphi}{ \sin\varphi}\right)=\alpha\left(\frac{u}{s},\varphi\right),
\end{align*}
we conclude that \eqref{seccases} is satisfied if and only if
$$
(\theta,t)\in \left( [-\frac{\pi}{2},\alpha\left(\frac{u}{s},\varphi\right))\times [0,u\cos\theta]\right)\cup \left(\big[\alpha\left(\frac{u}{s},\varphi\right), \varphi+\frac{\pi}{2}\big]\times  [0, s\cos(\theta- \varphi)]\right), 
$$
as illustrated in Figure \ref{figure4_5}.
Now the integral can be easily computed.
\begin{figure}
\centering
\includegraphics[width=400pt]{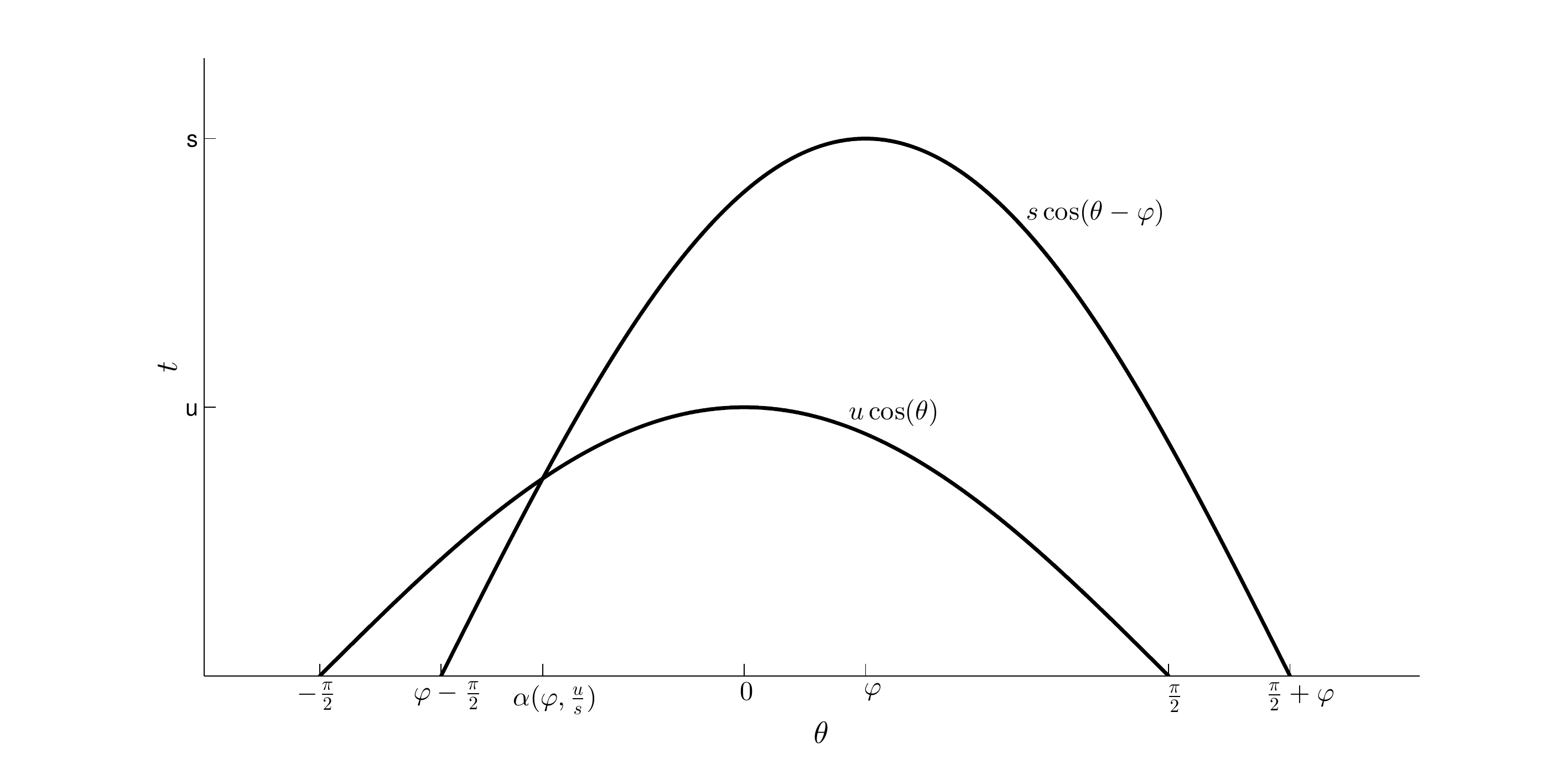}
\caption{Support of the indicator function for $\varphi=\frac{\pi}{7}$.}  \label{figure4_5}
\end{figure}

\end{proof}

\bigskip

%

\section{Variance inequalities}

In the next theorem, inequalities for $\Var[V_n(Z_0)]$ are provided. In these inequalities two auxiliary quantities, 
   $D(n,r)$ and $E(n,r)$, to be defined below, are involved. 
   In Lemma \ref{3.12} we establish crucial bounds for $E(n,r)$. 

\bigskip

For $n\in\mathbb{N}$ with $n\ge 2$ and $r>0$, we define
\[
D(n,r) := \frac{n\kappa_n^2}{r}\, \Gamma\left(\frac{2n}{r}+1\right)\,\left(\frac{n\kappa_n r}{4\gamma c(n,r)}\right)^{\frac{2n}{r}}
\]
and
\begin{align*}
 E(n,r) &:=  \frac{n}{c(2,n-2)} \int\limits_0^\pi  (\sin\varphi)^{n-2} \int\limits_0^1 t^{n-1}\left( 1+t^r-F_r(t,\varphi) \right) \,dt\,d\varphi.
\end{align*}
Introducing for $v \in [-{\pi}/{2},{\pi}/{2}]$ and $r \in (0,\infty)$ the function  $M(v,r)$ by 
\[
M(v,r) := \frac{1}{c(2,r)}\int\limits_v^{\frac{\pi}{2}} (\cos\theta)^r \,d\theta,
\]
we obtain
\begin{align}\label{late1}
 E(n,r) &= \frac{n}{c(2,n-2)} \int\limits_0^\pi  (\sin\varphi)^{n-2} \int\limits_0^1 t^{n-1} \big[t^r M(\alpha(t,\varphi),r)\\
&\qquad\qquad\qquad\qquad
 + M(\varphi-\alpha(t,\varphi),r)\big] \,dt\,d\varphi,\nonumber
\end{align}
where $-\pi/2<\alpha(t,\varphi)<\pi/2$ and $0< \varphi/2\le \varphi-\alpha(t,\varphi)\le \pi/2$ for $\varphi\in(0,\pi)$.
\begin{theorem}\label{3.10}
With these definitions, we have
\begin{align*}
E(n,r)\,D(n,r) \;\leq\; \Var[V_n(Z_0)]\; \leq \;  E(n,r)\, D(n,r)\,4^{\frac{2n}{r}+1}.
\end{align*}
\end{theorem}
\begin{proof}
The mean value theorem and \eqref{Fbound} yield that
\begin{align*}
& \int\limits_0^\pi \int\limits_0^1 \left(\frac{1}{F_r(t,\varphi)^{\frac{2n}{r}}} - \frac{1}{(t^r+1)^{\frac{2n}{r}}}\right) t^{n-1}(\sin \varphi)^{n-2}\,dt\,d\varphi\\
& \qquad \geq \frac{n}{r}\frac{1}{2^{\frac{2n}{r}}} \int\limits_0^\pi \int\limits_0^1 \Big(t^r+1-F_r(t,\varphi) \Big) t^{n-1}(\sin \varphi)^{n-2}\,dt\,d\varphi,
\end{align*}
hence we deduce from Theorem \ref{2}, that 
\[
\Var[V_n(Z_0)]\geq E(n,r) \frac{c(2,n-2)}{r\, 2^{\frac{2n}{r}}} \frac{8\pi b_{n,2}}{r} \Gamma\left(\frac{2n}{r}\right) \left(\frac{n\kappa_n r}{2\gamma c(n,r)}\right)^{\frac{2n}{r}}= E(n,r)\, D(n,r).
\]
Again by the mean value theorem and \eqref{Fbound}, we have
\begin{align*}
& \int\limits_0^\pi \int\limits_0^1 \left(\frac{1}{F_r(t,\varphi)^{\frac{2n}{r}}} - \frac{1}{(t^r+1)^{\frac{2n}{r}}}\right) t^{n-1}(\sin \varphi)^{n-2}\,dt\,d\varphi\\
& \qquad\leq 4\frac{n}{r}\, 2^{\frac{2n}{r}} \int\limits_0^\pi \int\limits_0^1 \Big(t^r+1-F_r(t,\varphi) \Big) t^{n-1}(\sin \varphi)^{n-2}\,dt\,d\varphi.
\end{align*}
Now the assertion follows in a similar way as for the lower bound.
\end{proof}

\bigskip

The next lemma provides upper and lower bounds for the auxiliary quantity $E(n,r)$. In the sequel, these bounds 
will then be combined with Theorem \ref{3.10} to obtain further consequences stated in Corollaries \ref{5.1} and \ref{5.6} and in Theorem \ref{5.10}.

\bigskip

\begin{lemma}\label{3.12}
\begin{enumerate}
\item[{\rm (a)}] For all $r \in (0,\infty)$, we have
$$
c_r \leq E(n,r) \leq {3}/{2}
$$
with a constant $c_r>0$ which depends on $r$ but not on $n$.
%
\item[{\rm (b)}] For all $r \in (0,\infty)$ and $n\geq 3$, we have
\begin{align*}
&c\,\frac{\left(1+\frac{r}{n}\right)^{-\frac{1}{2}}}{\sqrt{r+1}}\,2^{\frac{n}{2}}\,\left(1+\frac{r}{2n}\right)^{-\frac{n}{2}}
\left(1+\frac{n}{n+r}\right)^{-\frac{n+r}{2}}\\
&\qquad\qquad \leq 
E(n,r)\leq C \,\frac{(1+\frac{r}{n})}{\sqrt{r+1}} \,2^{\frac{n}{2}}\, \left(1+\frac{r}{2n}\right)^{-\frac{n}{2}}\left(1+\frac{n}{n+r}\right)^{-\frac{n+r}{2}} 
\end{align*}
with constants $c, C>0$ which  are independent of $r$ and $n$. 
\item[{\rm (c)}]
For all $r\in(0,\infty)$, we have 
\[
c\, \frac{1}{(r+1)^{2}} \leq 
E(2,r) \leq  \frac{1}{\sqrt{r+1}}
\]
with a constant $c>0$ which  is independent of $r$.

\end{enumerate}
\end{lemma}
\begin{proof}
%
%
(a) Let  $r \in (0,\infty)$  be fixed. By \eqref{Fbound} and Remark \ref{3.9}, (2), we get
\[
E(n,r)
\leq  \frac{n}{c(2,n-2)} \int\limits_0^\pi  (\sin\varphi)^{n-2} \int\limits_0^1 t^{n-1} \,\frac{3}{2} \,dt\,d\varphi = \frac{3}{2}.
\]
Next we bound $E(n,r)$ from below. For this we start from \eqref{late1}, use the fact that $M$ is nonnegative and decreasing with respect to its first argument and that $\alpha(t,\varphi)\leq {\varphi}/{2}\leq {\pi}/{4}$, 
for $\varphi\in (0,\pi/2)$,  by Remark \ref{3.9}, (1), and apply then Remark \ref{3.9}, (2). This leads to
\begin{align*}
E(n,r)
& \geq \frac{n}{c(2,n-2)} \int\limits_0^{\frac{\pi}{2}}  (\sin\varphi)^{n-2} \int\limits_0^1 t^{n+r-1} M(\alpha(t,\varphi),r)\, dt\, d\varphi\\
& \geq \frac{n}{c(2,n-2)} \int\limits_0^{\frac{\pi}{2}} (\sin\varphi)^{n-2}\,d\varphi  
\cdot\frac{M(\frac{\pi}{4},r)}{n+r}\\
& \ge (2(1+r))^{-1} M\left(\frac{\pi}{4},r\right)=c_r>0.
\end{align*}
%
%

\bigskip

\noindent
(b)
Note that
\[
E(n,r) = \frac{n}{c(2,n-2)c(2,r)} \int\limits_{0}^\pi (\sin\varphi)^{n-2} \int\limits_0^1 t^{n-1} 
h_\varphi(t) \,dt\, d\varphi,
\]
where  $h_\varphi: [0,1]\rightarrow \mathbb{R}$, for $\varphi\in(0,\pi)$, is defined by
\[
h_\varphi(t):= c(2,r)(1+t^r-F_r(t,\varphi)),\qquad t\in [0,1].
\]
Observe that $h_\varphi$ depends on $r$, although this dependence is not made explicit by our notation (since $r>0$ 
is arbitrary but fixed in this part of the proof). 
Recall that $-\pi/2<\alpha(t,\varphi)<\pi/2$, $0\leq \varphi-\alpha(t,\varphi)\leq {\pi}/{2}$ and note that 
\[
h_\varphi(t)= t^r \int\limits_0^{\frac{\pi}{2}-\alpha(t,\varphi)} (\sin\theta)^r \,d\theta + \int\limits_0^{\frac{\pi}{2}-(\varphi-\alpha(t,\varphi))}(\sin\theta)^r \,d\theta.
\]
 The function $h_\varphi$ is strictly increasing on $[0,1]$, since
\begin{align*}
h^\prime_\varphi(t)&= r t^{r-1}\int\limits_0^{\frac{\pi}{2}-\alpha(t,\varphi)}(\sin\theta)^r\, d\theta + \alpha_t(t,\varphi)\left( (\cos(\alpha(t,\varphi)-\varphi))^r - t^r (\cos(\alpha(t,\varphi))^r\right)\\
& = r t^{r-1} \int\limits_0^{\frac{\pi}{2}-\alpha(t,\varphi)} (\sin\theta)^r \,d\theta >0,
\end{align*}
where $0< {\pi}/{2}-\alpha(t,\varphi)< \pi$. Here, $\alpha_t$ denotes the partial derivative of $\alpha(t,\varphi)$ with respect to the first argument $t$. For the last equality we used that $\cos(\varphi)+\tan(\alpha(t,\varphi))\sin\varphi=t$, by the definition of $\alpha(t,\varphi)$, and therefore 
\begin{align*}
\cos(\alpha(t,\varphi)-\varphi) = \cos(\alpha(t,\varphi))\cos(\varphi) + \sin(\alpha(t,\varphi))\sin\varphi
 = t \cos(\alpha(t,\varphi)). 
\end{align*}
Since $\alpha(1,\varphi)={\varphi}/{2}$, we obtain for $t\in [0,1]$ that
\[
0= h_\varphi(0) \leq h_\varphi(t)\leq h_\varphi(1)=2\int\limits_0^{\frac{\pi-\varphi}{2}} (\sin\theta)^r
\, d\theta,
 \]
 which implies
 \begin{equation}\label{eqas1}
E(n,r)\leq \frac{2}{c(2,n-2)c(2,r)}\int\limits_0^\pi (\sin\varphi)^{n-2} \int\limits_0^{\frac{\pi-\varphi}{2}}(\sin\theta)^r \,d\theta\, d\varphi.
 \end{equation}
Reproducing the argument at the bottom of p. 925 in \cite{AS}, for $\varphi\in (0,\pi)$ we get that
\begin{align}\label{eqas2}
\int\limits_0^{\frac{\pi-\varphi}{2}}(\sin\theta)^r \,d\theta & \leq
\int\limits_0^{\frac{\pi-\varphi}{2}} \frac{\cos\theta}{\cos(\frac{\pi-\varphi}{2})}(\sin\theta)^r \,d\theta\nonumber\\
& = \frac{(\sin(\frac{\pi-\varphi}{2}))^{r+1}}{(r+1) \sin\frac{\varphi}{2}}\nonumber\\
&= \frac{(\cos\frac{\varphi}{2})^{r+1}}{(r+1)\sin\frac{\varphi}{2}}.
 \end{align}
From \eqref{eqas1} and \eqref{eqas2} we deduce for $n\geq 3$ that
\begin{align*}
 E(n,r) &\leq \frac{2}{(r+1)c(2,n-2)c(2,r)} \int\limits_0^\pi (\sin\varphi)^{n-2} \frac{(\cos\frac{\varphi}{2})^{r+1}}{\sin\frac{\varphi}{2}}\,d\varphi\\
 & \leq \frac{2^{n}}{(r+1)c(2,n-2)c(2,r)} \int\limits_0^{\frac{\pi}{2}} (\sin\psi)^{n-3}(\cos\psi)^{n+r-1}\,d\psi\\
 & = \frac{2^{n-1} \Gamma(\frac{n-2}{2})\Gamma(\frac{n+r}{2})}{(r+1)c(2,n-2)c(2,r) \Gamma(n+\frac{r}{2}-1)}\\
 &\leq \frac{2+\frac{r+2}{n-2}}{(r+1) \pi}\, 2^n \,\frac{\Gamma(\frac{n}{2})^2 \Gamma(\frac{n+r}{2})\Gamma(\frac{r}{2}+1)}{\Gamma(\frac{n-1}{2})\Gamma(\frac{2n+r}{2})\Gamma(\frac{r+1}{2})}.
 \end{align*}
Then Stirling's formula implies the inequality
 \[
E(n,r)\leq C\, \frac{1+\frac{r}{n}}{\sqrt{r+1}} \,2^{\frac{n}{2}}\, \left(1+\frac{n}{n+r}\right)^{-\frac{n+r}{2}} \left(1+\frac{r}{2n}\right)^{-\frac{n}{2}}. 
 \]

To derive a lower bound for $E(n,r)$, for all $n\ge 2$, we start from \eqref{late1} to get
$$
E(n,r)\ge \frac{n}{c(2,n-2)}\int\limits_0^\pi (\sin\varphi)^{n-2}\int\limits_0^1t^{n+r-1} M(\alpha(t,\varphi),r)\, dt\, d\varphi,
$$
since $M$ is nonnegative. 
For $\varphi\in (0,\pi)$, we have $\alpha(t,\varphi)\le{\varphi}/{2}$ and ${\varphi}/{2}\in (0,{\pi}/{2})$, hence $M({\varphi}/{2},r)\le M(\alpha(t,\varphi),r)$. Thus we get 
\begin{align*}
M(\alpha(t,\varphi),r)&\ge M\left(\frac{\varphi}{2},r\right)=\frac{1}{c(2,r)}\int\limits_{\frac{\varphi}{2}}^{\frac{\pi}{2}}(\cos \theta)^r\, d\theta\\
&\ge \frac{1}{c(2,r)}\int\limits_{\frac{\varphi}{2}}^{\frac{\pi}{2}}(\cos \theta)^r\sin\theta \, d\theta\\
&=\frac{1}{c(2,r)}\frac{1}{r+1}\left(\cos\frac{\varphi}{2}\right)^{r+1}.
\end{align*}
This implies that
\begin{align*}
E(n,r)&\ge \frac{n}{n+r}\frac{1}{r+1}\frac{1}{c(2,n-2)c(2,r)}\int\limits_0^\pi (\sin\varphi)^{n-2}
\left(\cos\frac{\varphi}{2}\right)^{r+1}\,d\varphi \\
&\ge c\, \frac{n}{n+r}\sqrt{\frac{n}{r+1}}\,2^{n-1}\,\int\limits_0^{\frac{\pi}{2}}(\sin\psi)^{n-2}(\cos\psi)^{n+r-1}\, d\psi \\
&= c\, \left(1+\frac{r}{n}\right)^{-1}\sqrt{\frac{n}{r+1}}\,2^{n-2}\,\frac{\Gamma\left(\frac{n-1}{2}\right)
\Gamma\left(\frac{n+r}{2}\right)}{\Gamma\left(\frac{2n+r-1}{2}\right)}.
\end{align*}
Stirling's formula then yields 
\begin{align*}
E(n,r)&\ge c\, \frac{\left(1+\frac{r}{n}\right)^{-1}}{\sqrt{r+1}}2^{\frac{n}{2}}\left(1+\frac{r+1}{2(n-1)}\right)^{-\frac{n-1}{2}}
\left(1+\frac{n-1}{n+r}\right)^{-\frac{n+r}{2}} \\
&\ge c\, \frac{\left(1+\frac{r}{n}\right)^{-\frac{1}{2}}}{\sqrt{r+1}}\,2^{\frac{n}{2}}\,\left(1+\frac{r}{2n}\right)^{-\frac{n}{2}}
\left(1+\frac{n}{n+r}\right)^{-\frac{n+r}{2}}, 
\end{align*}
which gives the lower bound.

\bigskip

\noindent
(c) Since the proof of the lower bound in (b) works also for $n=2$ we immediately get the lower bound. 

To derive the upper bound 
for $n=2$, we first use \eqref{eqas1} to get
\begin{align*}
E(2,r)&\le \frac{2}{\pi\, c(2,r)}\int\limits_0^\pi\int\limits_0^{\frac{\pi-\varphi}{2}}(\sin \theta)^r\, d\theta\, d\varphi\\
&=\frac{4}{\pi\, c(2,r)}\int\limits_0^{\frac{\pi}{2}}\left(\frac{\pi}{2}-\theta\right)(\sin\theta)^r\, d\theta\\
&=\frac{2}{ c(2,r)}\int\limits_0^{\frac{\pi}{2}} \frac{2}{\pi}\,\vartheta(\cos\vartheta)^r\, d\vartheta.
\end{align*}
Since $2\vartheta/\pi \le\sin\vartheta$, for $\vartheta\in[0,\pi/2]$, we  further deduce that
$$
E(2,r)\le \frac{2}{ c(2,r)}\int\limits_0^{\frac{\pi}{2}}\sin\vartheta (\cos\vartheta)^r\, d\vartheta 
=\frac{2}{ c(2,r)}\frac{1}{r+1}.
$$
Using Stirling's formula, we find $c(2,r)\ge 2/\sqrt{r+1}$ which yields that $E(2,r)\le 1/\sqrt{r+1}$.

\end{proof}

\bigskip

In the following corollary, the bounds in \eqref{eq1} are obtained from Proposition \ref{3.3} 
by applying Stirling's formula, compare \eqref{calc1}, \eqref{calc2} and \eqref{calc3}. For the bounds in \eqref{eq2} we first  use Theorem \ref{3.10}. Then, Lemma \ref{3.12} (a) implies that the factor $E(n,r)$ is bounded from above and below by constants depending only on $r$, whereas the factor $D(n,r)$ is equal to the upper bound from Proposition \ref{3.3} for $k=2$ multiplied with $(n/r)(1/2)^{(2n)/r}$, which has already been considered to obtain \eqref{eq1}.

\begin{corollary}\label{5.1}
For $k \in\mathbb{N}$ and fixed $r \in (0,\infty)$, there are constants $c_r,C_r > 0$, depending on $r$ and $k$ but not on $n$ or $\gamma$, 
such that
\begin{equation}\label{eq1}
c_r\, \left(A(r)\frac{n}{\gamma}\, \left(1+\frac{r}{n}\right)^{\frac{n}{2}}\right)^{\frac{kn}{r}}
\leq \mathbb{E}[V_n(Z_0)^k]
\leq C_r\,n^{\frac{1-k}{2}}\, \left(A(r)\frac{kn}{\gamma}\, \left(1+\frac{r}{n}\right)^{\frac{n}{2}} \right)^{\frac{kn}{r}},
\end{equation}
and  there are constants $c_r,C_r > 0$, depending on $r$  but not on $n$ or $\gamma$, such that
\begin{equation}\label{eq2}
c_r\, \sqrt{n}\left(A(r)\frac{n}{\gamma} \left(1+\frac{r}{n}\right)^{\frac{n}{2}}\right)^{\frac{2n}{r}} \leq  \Var[V_n(Z_0)] \leq C_r\, \sqrt{n}\left(A(r)\frac{4n}{\gamma} \left(1+\frac{r}{n}\right)^{\frac{n}{2}}\right)^{\frac{2n}{r}},
\end{equation}
where
\[
A(r):=
\frac{\pi^{\frac{r+1}{2}}}{e \Gamma(\frac{r+1}{2})}.
\]

\end{corollary}

\bigskip

For constant intensity $\gamma$, we infer from Corollary \ref{5.1} that $\mathbb{E}[V_n(Z_0)^k]$ and $\Var[V_n(Z_0)]$ 
go to infinity for fixed distance exponent $r$ and $n\to\infty$. 
Therefore we now choose the intensity $\gamma$  as a function of $n$ and $r$  such   that $\mathbb{E}[V_n(Z_0)]$ is equal to a positive constant $\lambda^{-1}$.
By Proposition \ref{3.3} we have
\[
\mathbb{E}[V_n(Z_0)] = \Gamma\left(\frac{n}{r}+1\right)\kappa_n \left(\frac{n\kappa_n r}{2\gamma c(n,r)}\right)^{\frac{n}{r}}.
\]
Therefore, if we define 
\[
\widehat{\gamma}(r,n) := \frac{n\kappa_n r}{2 c(n,r)} \left(\lambda \Gamma\left(\frac{n}{r}+1\right) \kappa_n \right)^{\frac{r}{n}}
\]
with $\lambda > 0$, then 
\[
\mathbb{E}[V_n(Z_0)] = \frac{1}{\lambda}.
\]
%
%
If we plug $\widehat{\gamma}(r,n)$ into the lower estimate from Theorem \ref{3.10}, keep in mind that by Lemma \ref{3.12} (a) the factor $E(n,r)$ is bounded from below by a constant depending only on $r$ and apply Stirling's formula, then we obtain that there is a constant $c_r>0$ depending on $r$ but not on $n$ or $\lambda$, such that
\[
\Var[V_n(Z_0)]\geq c_r\, \frac{1}{r}  \frac{n}{\lambda^2}\frac{\Gamma\left(\frac{2n}{r}+1\right)}{\Gamma\left(\frac{n}{r}+1\right)^2}2^{-\frac{2n}{r}}
\ge c_r \,\lambda^{-2}\sqrt{n} \rightarrow \infty \text{ as } n \rightarrow \infty.
\]

The preceding analysis suggests that in order to arrive at a limiting behaviour comparable to 
the case of a Poisson-Voronoi tessellation, i.e.\ with the variance converging to zero as $n\to\infty$, 
we cannot choose the distance parameter $r$ fixed but have to adjust it to the dimension. 
In fact, in the following theorem we consider the case where $r$ is proportional to $n$ which is 
the natural choice in view of the estimates that have been obtained.

\begin{corollary}\label{5.6}
Let $r = a n$ with  fixed $a\in (0,\infty)$ and $k\in\mathbb{N}$. Then there are constants $c,C >0$, depending on $a$ and $k$   but not on $n$ or $\gamma$, 
such that 
\begin{equation}\label{eq3}
c \;{\gamma^{-\frac{k}{a}}}\; n^{\frac{k}{a}-\frac{k}{2}} \left(\frac{B(a)}{n}\right)^{\frac{kn}{2}}
\leq
\mathbb{E}[V_n(Z_0)^k]
\leq
C \;{\gamma^{-\frac{k}{a}}}\; n^{\frac{k}{a}-\frac{k}{2}} \left(\frac{B(a)}{n}\right)^{\frac{kn}{2}},
\end{equation}
and there are constants $c,C >0$,   depending on $a$   but not on $n$ or $\gamma$, 
such that 
\begin{align}\label{eq4}
& c \, \left[\frac{4(a+1)^{a+1}}{(a+2)^{a+2}}\right]^{\frac{n}{2}}
{\gamma^{-\frac{2}{a}}}\; n^{\frac{2}{a}-\frac{3}{2}}\; \left(\frac{B(a)}{n}\right)^{n}\nonumber\\
& \qquad \qquad\leq  \Var[V_n(Z_0)]
\leq  C\, \left[\frac{4(a+1)^{a+1}}{(a+2)^{a+2}}\right]^{\frac{n}{2}}
{\gamma^{-\frac{2}{a}}}\; n^{\frac{2}{a}-\frac{3}{2}}\; \left(\frac{B(a)}{n}\right)^{n},
\end{align}
where
\[
B(a):= \frac{2\pi e(a+1)^{\frac{a+1}{a}}}{a}.
\]
Hence, if the intensity $\gamma$ is constant and $r=an$, then $\mathbb{E}[V_n(Z_0)^k]$ and $\Var[V_n(Z_0)]$ are converging to zero as $n\to\infty$.
\end{corollary}

\begin{proof}
The inequalities in \eqref{eq3} follow from Proposition \ref{3.3} by means of Stirling's formula, compare \eqref{calc1}. 

For the inequalities in \eqref{eq4} we first use Theorem \ref{3.10}. There is nothing to prove for $n=2$. Hence, we assume that $n\geq 3.$
 Then,  $E(n,an)$ is bounded by means of Lemma \ref{3.12} (b), whereas the factor $D(n,an)$ is equal to the upper bound from Proposition \ref{3.3} for $k=2$ multiplied by $(1/a)(1/2)^{2/a}$, which has already been considered to obtain \eqref{eq3}.
\end{proof}

\bigskip
Now we choose again $r=an$, with fixed $a>0$, and determine the intensity $\gamma$ in such a way that the expected  volume of the zero cell is equal to a positive constant.
This is possible for a special intensity function depending on  $a$, the dimension $n$ and a positive constant $\lambda$ that can be prescribed arbitrarily. In the following theorem, we describe the asymptotic behaviour 
of the volume of the zero cell in such a setting. 

\bigskip
\begin{theorem}\label{5.10}
Let $r=an$, with constant $a \in (0,\infty)$, and let the intensity be chosen as
\[
\widehat{\gamma}(a,n) = \frac{a n^2\kappa_n }{2 c(n,an)} \left(\lambda \Gamma\left(\frac{1}{a}+1\right) \kappa_n \right)^{a}
\]
with $\lambda > 0$. Then the following is true. 
\begin{enumerate}
\item[{\rm (a)}]
For $k\in\mathbb{N}$, 
\[
\frac{1}{\lambda^k}
\leq
\mathbb{E}[V_n(Z_0)^k]
\leq
\frac{\Gamma\big(\frac{k}{a}+1\big)}{\lambda^k \Gamma(\frac{1}{a}+1)^k};
\]
in particular, 
\[
\mathbb{E}[V_n(Z_0)] = \frac{1}{\lambda}.
\]
\item[{\rm (b)}] There are constants $c, C >0$, depending on $a$  but not on $n$ or $\lambda$, such that
\[
c\,\frac{1}{\lambda^2} \frac{1}{\sqrt{n}} \left[\frac{4(a+1)^{a+1}}{(a+2)^{a+2}}\right]^{\frac{n}{2}}
\leq 
\Var[V_n(Z_0)]
 \leq  C\, \frac{1}{\lambda^2} \frac{1}{\sqrt{n}} \left[\frac{4(a+1)^{a+1}}{(a+2)^{a+2}}\right]^{\frac{n}{2}};
\] 
in particular,
\[
\lim\limits_{n \rightarrow \infty}\Var[V_n(Z_0)] = 0.
\]
\end{enumerate}
\end{theorem}
\begin{proof}
The inequalities in (a) follow from Proposition \ref{3.3} and the special choice of the intensity as $\widehat{\gamma}(a,n)$.
In (b) there is nothing to prove for $n=2$. For $n\geq 3$ the inequalities follow from Theorem \ref{3.10}, since $D(n,an)$ is proportional to $\lambda^{-2}$ for the intensity chosen as $\widehat{\gamma}(a,n)$ and by using the bounds for $E(n,an)$ provided by Lemma \ref{3.12} (b).
\end{proof}

\bigskip
\begin{remark}\label{OrdnungWieAS}
For $a=1$, we have $\widehat\gamma(1,n)=n\kappa_n2^{n-1}\lambda$ and Theorem \ref{5.10} recovers the sharp bounds obtained by Alishahi and Sharifitabar in \cite{AS}.
\end{remark}

\begin{appendix}
\section{Auxiliary Inequalities}
The subsequent inequalities follow from Stirling's  formula for the Gamma function (see \cite[(12.33)]{WW} or \cite[p.~24]{Artin1}).
As before let $r \in (0,\infty)$. Then
\begin{align}\label{calc1}
& \frac{\sqrt{\pi}r}{\Gamma(\frac{r+1}{2}) (2e)^{\frac{r}{2}}}\left(\frac{n+r}{n}\right)^{\frac{n-1}{2}}(n+r)^{\frac{r}{2}} e^{-\frac{1}{6n}}\nonumber\\ 
& \qquad\qquad \leq \frac{n\kappa_n r}{2 c(n,r)} \leq \frac{\sqrt{\pi}r}{\Gamma(\frac{r+1}{2}) (2e)^{\frac{r}{2}}}\left(\frac{n+r}{n}\right)^{\frac{n-1}{2}}(n+r)^{\frac{r}{2}} e^{\frac{1}{6(n+r)}}.
\end{align}

Furthermore, for $k\in\mathbb{N}$, we have

\begin{align}\label{calc2}
& \sqrt{\frac{2\pi}{r e^2}} \left(\frac{e^2}{\pi}\right)^{\frac{k}{2}} \frac{\sqrt{kn+r}}{(n+2)^{\frac{k}{2}}}      
 \left(\frac{\sqrt{2e\pi}}{\left(\frac{re}{k}\right)^{\frac{1}{r}}}\right)^{nk}
 e^{-\frac{k}{6(n+2)}} \left(\left(n+\frac{r}{k}\right)^{\frac{1}{r}}(n+2)^{-\frac{1}{2}}\right)^{nk}\nonumber\\
 &  \leq \, \Gamma\left(\frac{kn}{r}+1\right)\kappa_n^k\nonumber\\
&  \leq \sqrt{\frac{2\pi}{r e^2}} \left(\frac{e^2}{\pi}\right)^{\frac{k}{2}} \frac{\sqrt{kn+r}}{(n+2)^{\frac{k}{2}}} \left(\frac{\sqrt{2e\pi}}{\left(\frac{re}{k}\right)^{\frac{1}{r}}}\right)^{nk}
e^{\frac{r}{12(kn+r)}} \left(\left(n+\frac{r}{k}\right)^{\frac{1}{r}}(n+2)^{-\frac{1}{2}}\right)^{nk}
\end{align}
and
\begin{align}\label{calc3}
&\Gamma\left(\frac{n}{r}+1\right)^k\kappa_n^k \nonumber\\
&\qquad \geq \left(\frac{2}{r}\right)^{\frac{k}{2}} \left(\frac{n+r}{n+2}\right)^{\frac{k}{2}} \left(\frac{\sqrt{2e\pi}}{(re)^{\frac{1}{r}}}\right)^{nk}
e^{-\frac{k}{6(n+2)}}\left((n+r)^{\frac{1}{r}} (n+2)^{-\frac{1}{2}}\right)^{nk}.
\end{align}

\end{appendix}
\section*{Acknowledgements}
The authors would like to thank a referee and an editor for their useful comments which helped to improve the manuscript and Walter Mickel (Karlsruhe Institute of Technology (KIT), Department of Mathematics, D-76128 Karlsruhe, Germany) for his  advice concerning the numerical calculations. 

\bigskip


\end{document}